\documentclass[a4paper,10pt]{article}
\usepackage[utf8]{inputenc}
\usepackage{amsmath,amssymb,amsthm}
\newtheorem{theorem}{Theorem}
\newtheorem{lemma}{Lemma}
\newtheorem{proposition}{Proposition}

\newtheorem{assumption}{Assumption}
\newtheorem{corollary}{Corollary}
\newtheorem{remark}{Remark}  

\newcommand{\R}{\mathbb{R}}

\renewcommand{\div}{{\rm div}}

\newcommand{\ltov}{L^2(\Omega)^n}

\newcommand{\ZD}{V_D}

\newcommand{\KK}{T} 
\newcommand{\low}{\leq_L} 
\newcommand{\gow}{\geq_L} 
\newcommand{\ai}{\underline{\alpha}}
\newcommand{\aii}{\overline{\alpha}}
\newcommand{\thet}{\theta_\eta}

\title{On the tangential cone condition for electrical impedance tomography}
\author{Stefan Kindermann\thanks{Industrial Mathematics Institute, Johannes Kepler University Linz, Austria. kindermann{@}indmath.uni-linz.ac.at.}}
\begin{document}
\maketitle
\begin{abstract}
We state some sufficient criteria for the tangential cone conditions to hold for the 
electrical impedance tomography problem. The results are based on an estimate 
for the first-order Taylor residual for the forward operator and some convexity
results in the L\"owner order. As a consequence, for conductivities satisfying
certain monotonicity properties, the tangential cone condition is verified.  
\end{abstract} 
\section{Introduction} 
The electrical impedance tomography problem is a classical inverse problem, 
where the aim is to extract information about the conductivity 
from boundary measurements of current/voltage pairs. 
Starting with  the definition of the problem in the 
seminal paper of Calderon \cite{Ca80}, it has been investigated in various direction 
and now serves as   paradigmatic instance of a parameter identification 
problem from boundary measurements.  

The common  mathematical formulation is to consider 
solutions of the boundary value problem on a Lipschitz domain $\Omega$, 
\begin{equation}\label{maineq}
\begin{split} 
\div (\gamma \nabla u) = 0 \quad \text{in } \Omega, \qquad 
u = f  \quad \text{ on } \Omega. 
\end{split}
\end{equation} 
The data for the inverse problem are multiple or infinitely many 
pairs of Cauchy-data $(f, \gamma \frac{\partial}{\partial n} u|_{\partial \Omega})$
on the boundary, and the interest is to  recover the conductivity $\gamma(x)$ in the 
interior $\Omega$. 
As it is typical for such 
identification problem with only boundary data, this leads under usual circumstances 
to a nonlinear severely  ill-posed problem, and without strong restrictions on the conductivity, 
one can at best only expect  conditional logarithmic stability \cite{Is98}. 
Various classical uniqueness and stability results are collected, e.g.,  in
\cite{Bo02} or  \cite{Is98}.

In the following, we assume that the unknown conductivity can be written as 
a perturbation $ \delta \gamma$ of a known background, which we take 
without loss of generality as $1$. Thus, we assume throughout that
\begin{equation}\label{ellip}
 \gamma(x) = 1 + \delta \gamma(x), \qquad \ai \leq \gamma(x) \leq \aii, \quad \text{a.e. in } \Omega, \end{equation}
with positive constants $\ai,\aii$ to ensure ellipticity and stability of the 
partial differential equation.  
 
In order to solve the problem, it is convenient to frame it into operator-theoretic 
language.  The above-mentioned Cauchy-data  (in the case of complete data),  are equivalent to 
the knowledge of the Dirichlet-to-Neumann operator.
  Thus, we
   introduce the parameter-to-data map:
\begin{equation}\label{defF} 
 F(\gamma):=   \Lambda_{\gamma} - \Lambda_1,  \end{equation} 
where $\Lambda_\gamma$ is the Dirichlet-to-Neumann map  for 
\eqref{maineq}. Solving the inverse problem is then equivalent to 
 inverting $F$.  We give a precise definition of  the associate spaces $X,Y$
in  $F:X\to Y$   in the next section. 
 
The main theme of this article concerns not  the solution of this problem but  
the investigation of the nonlinearity of $F$. This is highly relevant 
when applying regularization methods, in particular, iterative ones. 
Indeed, the convergence theory of iterative regularization methods 
such as Landweber's method requires some  restrictions on the nonlinearity
that quantify the deviation of the problem from a linear one.  
In this work, we focus on the well-known tangential cone 
conditions and of its variants (cf.~\cite{S95,HNS95,Ki17}). 

For a general inverse problem with a differentiable parameter-to-data map $F$ between Hilbert spaces, 
the  so-called strong tangential cone condition \cite{S95,HNS95} is
satisfied if, with  $F'$ denoting
the Fr\'{e}chet-derivative of $F$,  there is an $\eta$,  
 $1 > \eta>0$, such that 
\begin{equation}
\label{stc} 
 \|F(\tilde{x} ) - F(x) - F'(x) (\tilde{x} - x) \| \leq \eta \|F(\tilde{x}) - F(x)\|
\tag{stc}  
 \end{equation}  
holds for all  $\tilde{x},x$ in a neighborhood of some $x_0$. 
A weaker version, the weak-tangential cone conditions \cite{S95} holds if 
an $\eta$, $1 > \eta>0$, exists such that 
\begin{equation}
\label{wtc} 
 \left( F(\tilde{x} ) - F(x) - F'(x) (\tilde{x} - x), 
 F(\tilde{x}) - F(x) \right)_Y 
  \leq \eta \|F(\tilde{x}) - F(x)\|^2. 
\tag{wtc}  
 \end{equation}  
Moreover, the weak tangential cone condition with $\eta =1$ reads as 
    \begin{equation}
\label{qcon} 
 \left(  F'(x) (\tilde{x} - x), 
 F(\tilde{x}) - F(x) \right)_Y
   \geq 0,  
\tag{qcon}  
 \end{equation}
which yields a weaker condition than \eqref{wtc} that has been proposed in \cite{Ki17}. Note that by the
 parallelogram identity, the 
 inequalities \eqref{wtc} and \eqref{qcon} 
may be equivalently rewritten as 
\begin{equation} \label{para}
\begin{split} 
&\|F(\tilde{x} ) - F(x) - F'(x) (\tilde{x} - x)\|^2\\
 &  \qquad  \leq  (2 \eta -1) \|F(\tilde{x}) - F(x)\|^2 +  \|F'(x) (\tilde{x} - x)\|^2. 
 \end{split} 
 \end{equation}

These inequalities  are central to the convergence theory of the nonlinear Landweber method and 
many other iterative regularization methods. They are a replacement to  coercivity estimates, 
which cannot exist in the ill-posed case. 
It is a classical result that (under some standard additional assumptions)  
the strong tangential cone condition with $\eta \leq \frac{1}{2}$
implies strong convergence of the Landweber method
\cite{HNS95,KaNeSc08}.
Similarly, the weak-tangential cone condition   \cite{S95}
imply nonexpansivity of the iteration, and in particular, weak (subsequential) convergence of the iterates. (Of course, all this in connection with parameter choice rules.)
The condition \eqref{qcon} imply that the iterates stay in a neighborhood of the solution, 
which also yields weak (subsequential) convergence \cite{Ki17}. 

It is surprising that in view of its importance,  the tangential cone conditions for the impedance tomography problem
could only be verified in a few especial cases. For instance, Lechleitner and 
Rieder \cite{LeRi08} 
have proven \eqref{stc} in a semidiscrete case, 
essentially by 
using a stability result for the discrete problem. Interestingly, de Hoop, Liu and 
Scherzer \cite{HQS12}
have proven \eqref{stc} for a class of piecewise constant conductivities being constant on finitely many 
regions. 
  The proof is based on 
a Lipschitz stability result of Alessandrini and Vessella~\cite{AV05}.  In both cases, the 
stability constants might get quite large, thus, the cone conditions can in practice only 
be theoretically verified in a very narrow neighborhood.

Except from these few cases, the validity of the above tangential cone conditions 
is completely open, which is extremely puzzling given the fact that 
the Landweber method has successfully been applied 
to the impedance tomography problem in many situations, and only inveterate sceptics would  
doubt  its convergence. 

Our article aims to gain further understanding of this puzzle (though without completely resolving it) 
 by analyzing and establishing 
sufficient conditions for the tangential cone conditions.
The main contribution is that a condition of the form 
\[ \|F'[\gamma^\dagger] (\gamma-\gamma^\dagger)^2\| \leq C
 \|F'[\gamma^\dagger] (\gamma-\gamma^\dagger)\| \]
 suffices; see~\eqref{mjmi}. This is established by a useful 
 bound for the remainder in a first-order Taylor expansion in \eqref{mainest1}. 

 One result that is
probably most relevant in practice is that the tangential cone conditions are 
satisfied for conductivities that satisfy certain monotonicity properties
(e.g., a purely positive perturbation of the background conductivity).  
A central tool in this paper,  which might be of more general interest,
 is  a convexity result of the parameter-to-solution 
map in the L\"owner ordering given in Theorems~\ref{main1} and \ref{main2}.

\section{Problem setup and operator estimates}
We formulate some standard assumption and specify the notation. 
We make use of the space of $L^2$-vector fields: 
\[  \ltov:= \{v:\Omega \to \R^n\,|\, \int_\Omega |v(x)|^2 < \infty \}. \]
The Dirichlet data $f$ are canonically chosen in $H^\frac{1}{2}(\partial \Omega)$, and 
thus, the associated boundary value problem
\[ \Delta u_{1,f} = 0 \text{ in }\Omega, \qquad  u_{1,f} = f \quad \text{on }\partial \Omega \]
has a unique solution in $H^1(\Omega)$. Moreover, by the Poincar\'{e} inequality, 
the mapping $f \to \nabla u_{1,f}$ is an isomorphism from  
$H^\frac{1}{2}(\partial \Omega) \to \ltov$, such that we may define 
the $H^\frac{1}{2}(\partial \Omega)$-norm in this paper as 
\[ \|f\|_{H^\frac{1}{2}(\partial \Omega)} := \|\nabla u_{1,f}\|_{\ltov}. \]
Moreover, we define  $u_{\gamma,f}$ as the solution in $H^1(\Omega)$ of the problem \eqref{maineq} 
\[ \div (\gamma \nabla  u_{\gamma,f})= 0 \text{ in }\Omega, \qquad  u_{\gamma,f} = f \quad \text{on }\partial \Omega. \]

For a conductivity $\gamma$ satisfying \eqref{ellip}, 
the Dirichlet-to-Neumann operator $\Lambda_{\gamma}$
 maps a $f \in H^\frac{1}{2}(\partial \Omega)$ to 
$\gamma \frac{\partial}{\partial n} u_{\gamma,f} \in H^{-\frac{1}{2}}(\partial \Omega)$ continuously.
It is well-known that the Dirichlet-to-Neumann operators can be rewritten in terms of 
energy integrals: Defining $F$ by \eqref{defF}, we have 
 for $\gamma_1,\gamma_2$ satisfying \eqref{ellip} that (cf.~\cite{Is98}) 
\begin{align}\label{defFF}  F(\gamma_1) - F(\gamma_2) &= \Lambda_{\gamma_1} - \Lambda_{\gamma_2} \qquad \text{ with} \\
\langle \Lambda_{\gamma_1} - \Lambda_{\gamma_2} f, g \rangle &= 
\int_{\Omega} (\gamma_1 - \gamma_2) \nabla u_{\gamma_1,f}.\nabla u_{\gamma_2,g} dx. 
\end{align} 
Here and in the following we denote by $\langle. , .\rangle$ the duality paring in 
$H^{-\frac{1}{2}} \times H^{\frac{1}{2}}$:
\[  \langle  g , f \rangle  = \langle g,f \rangle_{H^{-\frac{1}{2}}(\partial \Omega) ,H^{\frac{1}{2}}(\partial \Omega)}  . \]
Moreover, the mapping $F(\gamma)$ is Fr\'{e}chet-differentiable  with respect to the $L^\infty$-norm, 
and the derivative can be expressed as 
\begin{align} \label{defFp} F'[\gamma_1] w &= \Lambda_{\gamma_1}'(w) \\
 \langle \Lambda_{\gamma_1}'(w) f,g \rangle &= 
\int_{\Omega} w \nabla u_{\gamma_1,f}.\nabla u_{\gamma_1,g} dx,  \qquad w \in L^\infty(\Omega).
\end{align}

Our analysis is
based on the following assumptions, which we assume to hold for the rest of the article. 
\begin{assumption} 
\item We assume given a (finite or countably infinite) sequence of 
linear independent and orthogonal 
Dirichlet data $(f_i)_{i \in I}$, $f_i \in H^\frac{1}{2}(\partial \Omega)$. 
\item We set as domain of definition of the parameter-to-data map
\[ D(F) := \{ \gamma \in L^\infty(\Omega)\,|\, \ai \leq \gamma \leq \aii, 
\text{ and } \eqref{next} \text{ holds} \}, \] 
with 
\begin{equation}\label{next} 
  \sum_{i,j \in I} 
\left| \langle[\Lambda_{\gamma} - \Lambda_1] f_i,f_j\rangle
\right|^2 < \infty. 
\end{equation} 
\end{assumption}
The first assumption is not much of a restriction, as the Dirichlet data are part of the 
experimental design and can be chosen  orthogonal. Also the second one is 
not severe, as 
\eqref{next} holds if we only have finitely many Dirichlet data, or 
in case of infinitely many $f_i$ if the deviations from the background conductivity
$\delta \gamma$ have a common compact  support inside of $\Omega$. This is also a usual assumption 
in the impedance tomography problem. We note that we do not a priori require that the
$f_i$ form a complete basis in $H^\frac{1}{2}(\partial \Omega)$. Thus, much of our analysis 
is also valid in case of finitely many measurements or measurements on 
only a part of the boundary.

Associated to the set of Dirichlet data, we set 
\[ \ZD:= \text{span}(f_i)_{i \in I} \subset H^\frac{1}{2}(\partial \Omega). \] 
The quadratic form 
$\langle (\Lambda_{\gamma} - \Lambda_1)f, g\rangle$
associated to the Dirichlet-to-Neumann mapping defines 
a linear operator $\Lambda_{\gamma} - \Lambda_1: \ZD \to \ZD'.$
Hence, the parameter to solution map can be defined as a mapping in the following spaces
\begin{align*} 
 F: D(F) \subset X := L^\infty(\Omega) &\to  Y := L(\ZD,\ZD') \subset
  L(H^\frac{1}{2}(\partial \Omega), H^{-\frac{1}{2}}(\partial \Omega))
  \\
  \gamma &\to 
 \Lambda_{\gamma} - \Lambda_1
 \end{align*} 
with the norm 
\begin{align*} 
 \|F(\gamma)\|_Y^2:&=  \sum_{i,j \in I} 
\left| \langle[\Lambda_{\gamma} - \Lambda_1] f_i,f_j\rangle
\right|^2.
\end{align*}
Introducing the Riesz isomorphism 
 $\mathcal{I}: H^{-\frac{1}{2}} \to H^{\frac{1}{2}}$, we may write the norm 
 \begin{align*} 
 \|F(\gamma)\|_Y^2 =  \sum_{i,j \in I}  \left| \left( \mathcal{I}[\Lambda_{\gamma} - \Lambda_1] f_i,f_j\right)_{H^{\frac{1}{2}},H^{\frac{1}{2}}}
\right|^2  = \|\mathcal{I}[\Lambda_{\gamma} - \Lambda_1] \|_{HS(\ZD)}^2 
\end{align*}
as  Hilbert-Schmidt norm  of 
 $\mathcal{I} (\Lambda_{\gamma} - \Lambda_1)$ for the operator mapping between  the space $\ZD$.
Thus, the image space is equipped with a Hilbert space structure.    
 Note that 
 for selfadjoint compact operators, the 
Hilbert-Schmidt norm is the sum of squares of the eigenvalues. 
We will denote  by $\|.\|_{HS}$ the Hilbert-Schmidt norm, omitting the underlying 
space where the operators map,  and 
by $\|.\|_{2,L(X,Y)}$ the operator norm for a linear operator from $X\to Y$.  
Moreover, for functions in $L^\infty$, we set $\|.\|_\infty = \|.\|_{L^\infty}$.

For later use we also recall the L\"owner ordering for selfadjoint 
operators on a Hilbert space $H$: We have 
\[ A \low B \Leftrightarrow  (Ax,x)_H \leq (Bx,x)_H \qquad \forall x \in H. \]
A useful property of this ordering is that for Hilbert-Schmidt operators, 
\begin{equation} \label{HSineq}
 0 \leq  A \low B \Rightarrow  \|A\|_{HS} \leq \|B\|_{HS}. \end{equation} 
This follows from Weyl's inequality for the eigenvalues and since 
the L\"owner ordering implies a corresponding eigenvalue inequality. 
Moreover, the following result with an arbitrary bounded operator $T$ 
and $T^*$ its adjoint 
 will be used frequently:
\begin{equation}\label{mult}
A \low B \Rightarrow T^* A T \low  T^* B T.
\end{equation}

We define the operator $L(\gamma)^{-1}$ as the solution operator 
for \eqref{maineq} with homogeneous Dirichlet condition, i.e., 
for $\gamma$ satisfying \eqref{ellip}, we set 
\begin{align*} 
 L(\gamma)^{-1}: H^{-1}(\Omega) &\to H_0^{1}(\Omega) \\ 
h &\to  v, 
\end{align*} 
where $v$ is the solution of 
\[ \div(\gamma \nabla v) = h \quad  \text{in } \Omega
  \qquad v = 0 \quad \text{on } \partial \Omega. \] 
Furthermore, for a function $\kappa \in L^\infty(\Omega)$, we define  the multiplication operator 
\begin{equation}\label{defmulti} 
 H_{\kappa} : \ltov \to \ltov \qquad \vec{f}(x) \to \kappa(x) \vec{f}(x)  \end{equation}
and the  bounded operator
\[ \KK_\gamma : \ltov  \to \ltov \qquad \KK_{\gamma}:= \nabla L(\gamma)^{-1}\div. \]
With this notation we state  the following useful lemma:
\begin{lemma}\label{lemma1}
For any $\gamma_1,\gamma_2$ with $\|\frac{\gamma_1-\gamma_2}{\gamma_2}\|_{\infty}<1$ 
satisfying \eqref{ellip} 
 and any $f \in H^\frac{1}{2}(\partial \Omega)$,  we have 
\[  \nabla u_{\gamma_1,f} = \left(I - \KK_{\gamma_2}H_{\gamma_2-\gamma_1}\right)^{-1}   \nabla 
u_{\gamma_2,f}, \]
\end{lemma} 
\begin{proof}
By definition of the inhomogeneous Dirichlet problem, we have that 
$u_{\gamma_2,f} = u_{\gamma_1,f} + w$, where $w$ satisfies the homogenous problem 
\[ \div (\gamma_2 \nabla w) = -\div(\gamma_2  u_{\gamma_1,f}) = 
-\div((\gamma_2 -\gamma_1) u_{\gamma_1,f}), \]
and where we used that $u_{\gamma_2,f}$ solves the problem \eqref{maineq} with $\gamma= \gamma_2$.
Thus, 
\[ u_{\gamma_2,f} = 
[I - L(\gamma_2)^{-1} \div((\gamma_2 -\gamma_1) \nabla] u_{\gamma_1,f}. \]
Applying the gradient yields 
\[ \nabla u_{\gamma_2,f} = 
[I - \nabla L(\gamma_2)^{-1} \div[(\gamma_2 -\gamma_1) ] \nabla u_{\gamma_1,f}. \]
 Under the stated assumptions, the operator on the left 
has a bounded inverse, as can also be seen from the next proposition. 
\end{proof} 
Note that the multiplication operator satisfies 
$H_{\kappa_1} H_{\kappa_2} = H_{\kappa_1 \kappa_2}$, these operators commute, 
and 
$\|H_\kappa\|_2 \leq \|\kappa\|_{\infty}$.

The operators $\KK_\gamma$ play an important role as they can be expressed 
as projection operator. 
\begin{proposition}\label{prop1}
The operator $\KK_\gamma$ can be written as 
\[ \KK_\gamma = H_{\gamma^\frac{1}{2}}^{-1} R_\gamma  H_{\gamma^\frac{1}{2}}^{-1}, \]
where $R_\gamma$ is an orthogonal projection operator on $\ltov$.
Moreover, the range of $R_\gamma$ is orthogonal to the space 
\begin{equation}\label{osp}
 \{ z \in \ltov \,|\, \div(\gamma^\frac{1}{2} z ) = 0 \}. \end{equation} 
\end{proposition} 
\begin{proof} 
Consider $R_\gamma =  H_{\gamma^\frac{1}{2}} \KK_\gamma H_{\gamma^\frac{1}{2}}$. 
We show that $R_\gamma^2 = R_\gamma$. Let $y \in \ltov$. Then by definition 
$R_\gamma y = \gamma^\frac{1}{2}\nabla v_y$, 
where $v_y \in H_0^1$ satisfies 
the variational formulation 
$$\int_\Omega \gamma \nabla v_y.\nabla w dx = \int_\Omega \gamma^\frac{1}{2} y. \nabla w dx$$ for 
arbitrary $w \in H_0^1$. Taking as $y = R_\gamma z =  \gamma^\frac{1}{2} \nabla v_z$, 
where $v_z$ satisfies the same equation as $v_y$ with $y$ replaced by $z$, we arrive at
\[ \int_\Omega \gamma \nabla v_y.\nabla w dx  = 
 \int_\Omega \gamma^\frac{1}{2} y. \nabla w dx
 = 
  \int_\Omega \gamma \nabla v_z. \nabla w dx = 
  \int_\Omega \gamma^\frac{1}{2} z. \nabla w dx, 
  \] 
   thus 
$ R_\gamma^2 y =  R_\gamma y$. 
 
Next, we show that  $R_\gamma$ is selfadjoint. Indeed, 
let $y,z \in \ltov$ be arbitrary and define $v_y,v_z$ as before. Then 
\[ (R_\gamma y, z)_{\ltov} = 
\int_\Omega \gamma^\frac{1}{2}\nabla v_y. z  = 
\int_\Omega \gamma \nabla v_y.\nabla v_z dx,
\]
which is symmetric in $y,z$, thus $R_\gamma y = R_\gamma^* y$, hence it is selfadjoint. 
As a consequence,  $R_\gamma$ is an orthogonal projector. 
Finally we show that $R_\gamma$ annihilates the space in \eqref{osp}. 
Let $z \in \ltov$ be in the space in \eqref{osp}. 
Then 
\[ R_\gamma z = \gamma^\frac{1}{2} \nabla L(\gamma^\dagger)^{-1}\div (\gamma^\frac{1}{2} z) = 0 \]
since $\div (\gamma^\frac{1}{2} z) = 0$ by definition. 
\end{proof}  

A useful observation is the following monotonicity property: 
\begin{lemma}\label{fprimepos}
If  $w(x) \geq 0$ a.e., $w \in L^\infty(\Omega)$,  then 
$ F'[\gamma] w \gow 0$. 
In particular, for $w_1,w_2 \in  L^\infty(\Omega)$ with $0 \leq 
w_1\leq w_2 $ a.e. in $\Omega$, and any $\gamma \in D(F)$, we have the estimate 
\begin{align}\label{uese} 
 \|F'[\gamma] (w_1 ) \|_Y \leq  \|F'[\gamma] (w_2) \|_Y .
\end{align} 
\end{lemma}
\begin{proof} 
For $f \in H^\frac{1}{2}(\partial \Omega)$ we have by \eqref{defFp} that 
\[ \langle  (F'[\gamma] w)f,f\rangle 
=  \langle \Lambda_\gamma'(w),f,f\rangle 
= \int_\Omega w(x) |\nabla u_{\gamma,f}|^2 dx \geq 0 .\] 
The second result follows now from \eqref{HSineq} since the assumption implies 
\[ 0 \low  \Lambda_\gamma'(w_1) \low   \Lambda_\gamma' (w_2). \]
\end{proof}

We can now state the important structural convexity property of the forward map. 
\begin{theorem}\label{main1}
For any $\gamma,\gamma^\dagger \in D(F)$
that satisfy  $\|\tfrac{\gamma-\gamma^\dagger}{\gamma^\dagger}\|_{\infty}<1$,
we have 
\begin{equation}\label{uppermain} 
0 \low \Lambda(\gamma)  - \Lambda(\gamma^\dagger) - \Lambda_\gamma'(\gamma-\gamma^\dagger) 
\low 
\Lambda_\gamma'(\tfrac{|\gamma-\gamma^\dagger|^2}{\gamma^\dagger}). 
\end{equation} 
\end{theorem} 
\begin{proof} 
Define 
\begin{equation}\label{meq}
 B(\gamma,\gamma^\dagger):= \Lambda(\gamma)  - \Lambda(\gamma^\dagger) - \Lambda_\gamma'(\gamma-\gamma^\dagger).
 \end{equation} 
Then, from \eqref{defFF} and \eqref{defFp}, it follows that 
\[ \langle B(\gamma,\gamma^\dagger)f, f\rangle = 
\int_\Omega (\gamma-\gamma^\dagger) \nabla u_{\gamma,f}.(\nabla u_{\gamma^\dagger,f}-\nabla u_{\gamma,f}) dx. \]
Using Lemma~\ref{lemma1} with $\gamma_2 = \gamma^\dagger$ and
$\gamma_1 = \gamma$
and \eqref{defmulti}, we write this as 
\begin{align} 
 &\langle B(\gamma,\gamma^\dagger)f, f\rangle = 
\left( H_{\gamma-\gamma^\dagger}\nabla u_{\gamma,f}, 
 \left[ (I - T_{\gamma_{\dagger}} H_{\gamma^\dagger-\gamma})-I \right]
\nabla u_{\gamma,f} \right)_{\ltov} \nonumber  \\
& = \left( H_{\gamma-\gamma^\dagger}\nabla u_{\gamma,f}, 
 T_{\gamma_{\dagger}} H_{\gamma-\gamma^\dagger}
\nabla u_{\gamma,f}\right)_{\ltov}  \label{newest}
\end{align} 
Since $T_{\gamma_{\dagger}}$ is positive definite by Proposition~\ref{prop1}, 
the last term is positive, which proves that
$B(\gamma,\gamma^\dagger) \gow 0$.

For the upper bound, we 
define the shortcut  
\[  \Delta \gamma:= \frac{\gamma-\gamma^\dagger}{(\gamma^\dagger)^\frac{1}{2}}. \]
Using  
Proposition~\ref{prop1},  
the fact that an orthogonal projector satisfies $P \low I $, and \eqref{mult}, 
we find that 
\begin{align}  
\langle B(\gamma,\gamma^\dagger)f, f\rangle &=  
\left( H_{\Delta \gamma}\nabla u_{\gamma,f}, 
R_{\gamma_{\dagger}} H_{\Delta \gamma}
\nabla u_{\gamma,f}\right)_{\ltov} \nonumber \\
&\leq 
 \left( H_{\Delta \gamma}\nabla u_{\gamma,f},  H_{\Delta \gamma}
\nabla u_{\gamma,f}\right)_{\ltov}  \label{thisthis}  \\
&= \int_{\Omega} \frac{(\gamma-\gamma^\dagger)^2}{(\gamma^\dagger)}
|\nabla u_{\gamma,f}|^2 dx  = 
\langle \Lambda_{\gamma}'\left((\Delta \gamma)^2\right)f,f\rangle.
\nonumber
\end{align} 
This verifies the upper bound. 
\end{proof} 

The upper bound can be strengthened by a more detailed analysis. 
The following result is an improvement. 
\begin{theorem} \label{main2} 
With the same assumption on $\gamma,\gamma^\dagger$ as in Theorem~\ref{main1},
let $B(\gamma,\gamma^\dagger)$ be as in \eqref{meq}, i.e., 
\[ B(\gamma,\gamma^\dagger) = \Lambda_{\gamma} - \Lambda_{\gamma^\dagger} 
-\Lambda_\gamma'(\gamma-\gamma^\dagger) \]
Then 
\[ \langle B(\gamma,\gamma^\dagger) f,f\rangle \leq  \inf_{w \in \ltov:
 \div ((\gamma^\dagger)^\frac{1}{2} w) = 0} \|H_{\Delta \gamma}
\nabla u_{\gamma,f}  - w \|^2. \]
Moreover,  the following estimates holds:
\begin{align} 
 B(\gamma,\gamma^\dagger)  &\low
 \Lambda_\gamma' (\tfrac{|\gamma-\gamma^\dagger|^2}{\gamma^\dagger})
 - [\Lambda_\gamma -\Lambda_{\gamma^\dagger} ] \Lambda_{\gamma^\dagger}^{-1} 
 [\Lambda_\gamma -\Lambda_{\gamma^\dagger} ], \label{util1} \\
  B(\gamma,\gamma^\dagger)  &\low
 \Lambda_\gamma' (\tfrac{|\gamma-\gamma^\dagger|^2}{\gamma^\dagger})
 -
  \Lambda_\gamma'\left((\gamma-\gamma^\dagger)\tfrac{\gamma}{\gamma^\dagger}  \right)
 \left[\Lambda_{\gamma}'(\tfrac{\gamma^2}{\gamma^\dagger}) \right]^{-1}
 \Lambda_\gamma'\left((\gamma-\gamma^\dagger)\tfrac{\gamma}{\gamma^\dagger}   \right). \label{util2} 
\end{align} 
\end{theorem} 
\begin{proof}
With \eqref{thisthis} we find that 
\[ 
 \langle B(\gamma,\gamma^\dagger)f, f\rangle = 
 \|R_{\gamma^\dagger} H_{\Delta \gamma} 
\nabla u_{\gamma,f} \|^2 
=  \inf_{w \in R_{\gamma^\dagger}^\bot} \|H_{\Delta \gamma}
\nabla u_{\gamma,f}  - w \|^2,  \]
where $R_{\gamma^\dagger}^\bot$ is the orthogonal complement of the range of 
$R_{\gamma^\dagger}$. By Proposition~\ref{prop1}, the set 
of $w$ with  $\div ((\gamma^\dagger)^\frac{1}{2} w)$ is subset of $R_{\gamma^\dagger}$, which 
yields the first inequality. 

Now we take $w = w_c$  as
\[ w_c = \sum_{i=1}^\infty c_i (\gamma^\dagger)^\frac{1}{2} \nabla u_{\gamma^\dagger,f_i}, \]
where 
$c_i \in \ell^2$ are coefficients to be specified 
below. Note that $\div ((\gamma^\dagger)^\frac{1}{2} w_c) = 0$ since 
$\nabla u_{\gamma^\dagger,f_i}$ solves  \eqref{maineq} with $\gamma = \gamma^\dagger$. 

 Expanding the square we find that 
\begin{align*} 
\|H_{\Delta \gamma}
\nabla u_{\gamma,f}  - w_c \|^2 = \int H_{\Delta \gamma}^2
|\nabla u_{\gamma,f}  |^2 dx &- 2 \sum_{i=1}^\infty c_i 
\int_{\Omega} (\gamma-\gamma^\dagger) \nabla u_{\gamma,f}. \nabla u_{\gamma^\dagger,f_i} dx \\
&+  \sum_{i,j=1}^\infty c_i c_j \int \gamma^\dagger   \nabla u_{\gamma^\dagger,f_i}.  
\nabla u_{\gamma^\dagger,f_i} dx \\
 = 
\langle \Lambda_{\gamma}'\left((\Delta \gamma)^2\right)f,f\rangle
&- 2 \sum c_i \langle \Lambda_{\gamma} - \Lambda_{\gamma^\dagger} f, f_i\rangle\\
&+ \sum_{i,j=1}^\infty c_i c_j \langle \Lambda_{\gamma^\dagger} f_i, f_j \rangle. 
\end{align*} 
Note that $f_i$ is an orthogonal basis such that $f_c =\sum c_i f_i$ represents an element 
in $H^\frac{1}{2}$  and hence, the last line can be written in terms of $f,f_c$ and 
the Dirichlet-to-Neumann operators. Now minimizing over $f_c$ yields that 
\begin{align*} 
\inf_c \|H_{\Delta \gamma}
\nabla u_{\gamma,f}  - w \|^2 = 
&\langle \Lambda_{\gamma}'\left((\Delta \gamma)^2\right)f,f\rangle\\
& \qquad - \langle \Lambda_{\gamma^\dagger}^{-1} (\Lambda_{\gamma} - \Lambda_{\gamma^\dagger}) f, 
(\Lambda_{\gamma} - \Lambda_{\gamma^\dagger}) f \rangle, 
\end{align*} 
which proves the first result. 

For the second one, we take 
\[ w_c = \sum_{i=1}^\infty c_i
 \frac{\gamma}{(\gamma^\dagger)^\frac{1}{2}} \nabla u_{\gamma,f_i}, \]
 which is again in $R_{\gamma^\dagger}$. As above, we find an analogous result with 
\begin{align*} 
&\|H_{\Delta \gamma}
\nabla u_{\gamma,f}  - w_c \|^2 \\
&= \int H_{\Delta \gamma}^2
|\nabla u_{\gamma,f}^2| dx - 2 \sum_{i=1}^\infty c_i 
\int_{\Omega} (\gamma-\gamma^\dagger)\frac{\gamma}{\gamma^\dagger}
 \nabla u_{\gamma,f}. \nabla u_{\gamma,f_i} dx \\
& \qquad \qquad \qquad 
+  \sum_{i,j=1}^\infty c_i c_j \int \frac{\gamma^2}{\gamma^\dagger} \nabla u_{\gamma,f_i}.  
\nabla u_{\gamma,f_i} dx  \\
& =
\langle \Lambda_{\gamma}'\left((\Delta \gamma)^2\right)f,f\rangle 
- 2\langle \Lambda_{\gamma}'\left((\gamma-\gamma^\dagger)\frac{\gamma}{\gamma^\dagger}\right) 
f, f_c\rangle
+  \langle \Lambda_{\gamma}'\left(\frac{\gamma^2}{\gamma^\dagger} \right)  f_c, f_c \rangle. 
\end{align*}
Proceeding as above yields the second result. 
\end{proof}



Next we derive some useful ordering and norm estimates: 
By rearranging terms and switching $\gamma$ and $\gamma^\dagger$ in \eqref{uppermain} using the lower bound, 
and then the upper bound in  \eqref{uppermain},
we obtain  the following two inequalities. 
\begin{corollary} 
Under the same assumptions as in Theorem~\ref{main1}, we have 
\begin{equation} \label{conmo}
 \Lambda_\gamma'(\gamma -\gamma^\dagger) \low 
\Lambda(\gamma) - \Lambda(\gamma^\dagger) \low   \Lambda_{\gamma^\dagger}'(\gamma -\gamma^\dagger)
\end{equation}
as well as 
\begin{equation} \label{conmo1}
\Lambda_{\gamma^\dagger}'\left(
\tfrac{\gamma^\dagger}{\gamma} (\gamma -\gamma^\dagger)\right) \low 
\Lambda(\gamma) - \Lambda(\gamma^\dagger) \low   \Lambda_\gamma'\left(
\tfrac{\gamma}{\gamma^\dagger} (\gamma -\gamma^\dagger)\right). 
\end{equation}
\end{corollary} 
Note that the lower bound in \eqref{uppermain} and consequently also in \eqref{conmo} was already been established by
von Harrach and Seo \cite{HaSe} and used in various uniqueness and stability estimates \cite{MR3900209,HaU13}. 
Our upper bound in \eqref{uppermain} seem to be new to the knowledge of the author.

We may observe by a Taylor expansion, that the middle term in the 
inequality \eqref{uppermain} is 
$$-\frac{1}{2} F''[\gamma](\gamma^\dagger-\gamma,\gamma^\dagger-\gamma) + 
o(|\gamma^\dagger-\gamma|^2) = 
-\frac{1}{2} F''[\gamma^\dagger](\gamma^\dagger-\gamma,\gamma^\dagger-\gamma), $$
while 
 the left-hand side is 
$ F'[\gamma^\dagger](\frac{|\gamma^\dagger-\gamma|^2}{\gamma^\dagger}) + 
o(|\gamma^\dagger-\gamma|^2)$. 
By setting $\gamma = \gamma^\dagger + \epsilon w$ for arbitrary $w \in L^\infty$, 
and taking the limit, we obtain 
\[ 0 \leq - F''[\gamma^\dagger](w,w) \leq 
2 F'[\gamma^\dagger](\tfrac{|w|^2}{\gamma^\dagger}). \] 
Hence the second derivative of $F$ is always negative definite. A similar definite result 
can be obtained for all even derivatives. 

For later use, we also establish a related inequality:
\begin{lemma}
Let $\xi_{\dagger}:= \|\frac{\gamma-\gamma^\dagger}{\gamma^\dagger}\|_\infty <1$
%
 Then we have 
the estimates 
\begin{align} 
0 \low 
 \Lambda_{\gamma^\dagger}'(\gamma - \gamma^\dagger) -
 \Lambda_\gamma'(\gamma - \gamma^\dagger) &\low (2 +\xi_\dagger)
  \Lambda_\gamma'\left(\tfrac{|\gamma-\gamma^\dagger|^2}{\gamma^\dagger}\right)
 \label{babel0}  
\end{align}
\end{lemma}
\begin{proof}
Define 
\[\mathcal{A} :=   \Lambda_{\gamma^\dagger}'(\gamma - \gamma^\dagger) -
\Lambda_\gamma'(\gamma - \gamma^\dagger) .\]
Then using Lemma~\ref{lemma1} with $\gamma_1 = \gamma$, $\gamma_2 = \gamma^\dagger$, 
we find  
\begin{align*} \langle \mathcal{A}f,f\rangle &= 
\int (\gamma-\gamma^\dagger) (|\nabla u_{\gamma^\dagger,f}|^2- 
|\nabla u_{\gamma,f}|^2) dx \\
& = \left(H_{\gamma-\gamma^\dagger} ( I + H_{(\gamma^\dagger)^{-\frac{1}{2}}} 
R_{\gamma^\dagger} H_{\Delta \gamma})\nabla u_{\gamma,f}, 
( I + H_{(\gamma^\dagger)^{-\frac{1}{2}}} 
R_{\gamma^\dagger} H_{\Delta \gamma})\nabla u_{\gamma,f}\right)\\
&\qquad - 
\left(H_{\gamma-\gamma^\dagger} \nabla u_{\gamma,f},\nabla u_{\gamma,f} \right) \\
& = \left( \left[2 H_{\Delta \gamma} R_{\gamma^\dagger} H_{\Delta \gamma}+ 
H_{\Delta \gamma}  R_{\gamma^\dagger} H_{\frac{\Delta \gamma}{(\gamma^\dagger)^\frac{1}{2}}}
R_{\gamma^\dagger}   H_{\Delta \gamma} \right] \nabla u_{\gamma,f},\nabla u_{\gamma,f} \right)\\
& = 
\left( \left[2 I +
H_{\frac{\Delta \gamma}{(\gamma^\dagger)^\frac{1}{2}}} \right] 
R_{\gamma^\dagger} H_{\Delta \gamma}
\nabla u_{\gamma,f},
R_{\gamma^\dagger} H_{\Delta \gamma} 
\nabla u_{\gamma,f} \right), 
\end{align*}
where we used that $R_{\gamma^\dagger}^2 = R_{\gamma^\dagger}$ and all operators 
are selfadjoint. Since 
$\xi_\dagger = \left\|\frac{\Delta \gamma}{(\gamma^\dagger)^\frac{1}{2}}\right\|$, we 
find that 
\[0 \low  (2 - \xi_\dagger) I \low 
2 I +
H_{\frac{\Delta \gamma}{(\gamma^\dagger)^\frac{1}{2}}} \low 
 (2 + \xi_\dagger) I,\]
and the lower bound follows as well as  the first upper bound since 
\begin{align*} 
&\left( R_{\gamma^\dagger} H_{\Delta \gamma}
\nabla u_{\gamma,f} ,
R_{\gamma^\dagger} H_{\Delta \gamma} 
\nabla u_{\gamma,f} \right) 
 = \|R_{\gamma^\dagger} H_{\Delta \gamma} \nabla u_{\gamma,f}\|^2 
\leq \| H_{\Delta \gamma} \nabla u_{\gamma,f}\|^2 \\
&= 
\langle \Lambda_\gamma'(\Delta \gamma)^2f,f\rangle.
\end{align*}
\end{proof}

As a consequence of Theorem~\ref{main1} and \ref{main2} we find the 
following upper bounds 
\begin{theorem}\label{th:three}
Under the same assumptions as in Theorem~1, we have 
\begin{align}\label{mainest1}
\| F(\gamma)  - F(\gamma^\dagger) - F'[\gamma](\gamma-\gamma^\dagger) \|_Y^2 \leq 
\|F'[\gamma]\left(\tfrac{|\gamma-\gamma^\dagger|^2}{\gamma^\dagger}\right) \|_Y^2.
\end{align}  
Moreover, there exists a constant $C$ depending only on $\ai,\aii,\Omega$ such that 
\begin{align}
\|F(\gamma)  - F(\gamma^\dagger) \|_Y^2 \leq C \sum_{i\in I} \left\langle  
\Lambda_\gamma'\left(\tfrac{|\gamma-\gamma^\dagger|^2}{\gamma^\dagger} \right) f_i,f_i \right\rangle  \label{x1} \\
 \|F'[\gamma](\tfrac{(\gamma-\gamma^\dagger)\gamma}{\gamma^\dagger})
 \|_Y^2 \leq C \sum_{i\in I}  \left\langle 
\Lambda_\gamma'\left(\tfrac{|\gamma-\gamma^\dagger|^2}{\gamma^\dagger} \right) f_i,f_i \right\rangle. \label{x2}
\end{align} 
\end{theorem} 
The last two upper bound can be interpreted as the trace norm of the operator
$\mathcal{I} \Lambda_\gamma'(\frac{|\gamma-\gamma^\dagger|^2}{\gamma^\dagger})$, 
i.e., the $\ell^1$-norm of the eigenvalues.  
\begin{proof}
The estimate \eqref{mainest1} is a direct consequence of \eqref{uppermain}. In view of \eqref{x1}, we find as
a consequence of \eqref{util1} and the fact that $\Lambda_{\gamma^\dagger}^{-1}$ is positive definite
 that 
\[ 0 \low 
  [\Lambda_\gamma -\Lambda_{\gamma^\dagger} ] \Lambda_{\gamma^\dagger}^{-1} 
 [\Lambda_\gamma -\Lambda_{\gamma^\dagger} ] \low  \Lambda_\gamma \left(( \Delta \gamma )^2 \right). \]
Standard elliptic estimates yield 
\[ \langle\Lambda_{\gamma^\dagger} f,f\rangle \sim
\|f\|_{H^\frac{1}{2}}^2 =  \langle\mathcal{I} f,f\rangle.    \]
where $\mathcal{I}$ is the Riesz isomorphism.  With $\Delta F:=   [\Lambda_\gamma -\Lambda_{\gamma^\dagger} ] $ we find that 
\[ \|\mathcal{I} (\Delta F) f\|_{H^\frac{1}{2}}^2  =  \langle (\Delta F), \mathcal{I} (\Delta F) f\rangle  \leq C
\langle \Lambda_\gamma' \left( (\Delta \gamma )^2 \right) f,f \rangle.  \] 
An equivalent alternative definition of the 
 Hilbert-Schmidt norm is the sum of the left-hand side over an orthogonal system $f_i$, which 
 yields the desired inequality \eqref{x1}. In the same way we obtain \eqref{x2}
 from \eqref{util2}.  
\end{proof} 

\begin{remark} \rm 
We note that the norm estimates \eqref{uese}, \eqref{mainest1}, \eqref{x1}, and \eqref{x2} remain 
valid, wenn the $Y$-norm is replaced by the operator norm $\|.\|_{L(H^\frac{1}{2}, H^{-\frac{1}{2}})}$.
This follows from the fact that the inequalities are derived 
from the L\"owner ordering and hold in particular for the largest eigenvalues, which agree with 
the operator norm. 
\end{remark}

\section{Tangential cone conditions}  
We can now state the first sufficient conditions for the 
tangential cone conditions. An immediate corollary of  the upper bound 
\eqref{mainest1} leads to the first condition:
\begin{theorem}
Let $\gamma,\gamma^\dagger \in D(F)$ with $\|\frac{\gamma-\gamma^\dagger}{\gamma^\dagger}\|_{\infty}<1$.
If 
 \[ 
\|F'[\gamma](\tfrac{|\gamma-\gamma^\dagger|^2}{\gamma^\dagger})\|_Y
\leq   
\|F'[\gamma](\gamma - \gamma^\dagger)\|_Y,\]
 then  the weak tangential cone condition with $\eta \geq \frac{1}{2}$  is satisfied. 

If for some   $\zeta <1$, it holds that 
 \[ 
\|F'[\gamma](\tfrac{|\gamma-\gamma^\dagger|^2}{\gamma^\dagger})\|_Y
\leq   \zeta
\|F'[\gamma](\gamma - \gamma^\dagger)\|_Y,\]
then the strong tangential cone condition 
with $\eta = \frac{\zeta}{1-\zeta}$ is satisfied. 
\end{theorem} 
\begin{proof}
In the second case, we have that 
\begin{align*} 
\|F'[\gamma](\gamma - \gamma^\dagger)\|_Y &\leq 
\|F'[\gamma](\gamma - \gamma^\dagger) - (F(\gamma) - F(\gamma^\dagger))\|_Y
+ \| F(\gamma) - F(\gamma^\dagger)\|_Y\\
& \leq 
\|F'[\gamma](\tfrac{|\gamma-\gamma^\dagger|^2}{\gamma^\dagger})\|_Y 
+ \|F(\gamma) - F(\gamma^\dagger)\|_Y \\
& \leq 
\zeta
\|F'[\gamma](\gamma - \gamma^\dagger)\|_Y 
+ \| F(\gamma) - F(\gamma^\dagger)\|_Y. 
\end{align*} 
From this, we conclude that 
\[ \|F'[\gamma](\gamma - \gamma^\dagger)\|_Y  \leq \frac{1}{1-\zeta} \| F(\gamma) - F(\gamma^\dagger)\|_Y \]
Thus, 
\begin{align*} 
 \|
(F(\gamma) - F(\gamma^\dagger)- F'[\gamma](\gamma - \gamma^\dagger) \|_Y
\leq \|F'[\gamma](\tfrac{|\gamma-\gamma^\dagger|^2}{\gamma^\dagger})\|_Y \\
\leq \zeta\|F'[\gamma](\gamma - \gamma^\dagger)\|_Y 
\leq \frac{\zeta }{1-\zeta} \| (F(\gamma) - F(\gamma^\dagger)\|_Y.
\end{align*} 
The weak cone condition for $\eta \geq \frac{1}{2}$ follows easily from  \eqref{para} and 
the upper estimate \eqref{uppermain}. 
\end{proof}

Since the above conditions involve der Fr\'{e}chet-derivative at  $\gamma$, it is of interest to 
derive localized conditions that only require $F'[\gamma^\dagger]$. 
This is established in the next theorem.  
\begin{theorem}\label{th:mjmi}
Let $\gamma,\gamma^\dagger \in D(F)$ with $\|\tfrac{\gamma-\gamma^\dagger}{\gamma^\dagger}\|_{\infty}<1$ and 
$ \|\tfrac{\gamma-\gamma^\dagger}{\gamma}\|_{\infty}<1$.
If it holds that 
 \begin{equation}\label{mjmi}
\|F'[\gamma^\dagger](|\gamma-\gamma^\dagger)|^2)\|_Y
\leq   \thet  
\|F'[\gamma^\dagger](\gamma - \gamma^\dagger)\|_Y,\end{equation}
with 
\[  \thet  
\leq \ai \frac{\eta}{4+ \eta},    \]
 then 
the strong tangential cone condition is satisfied with this $\eta$. 
Moreover, the estimate \eqref{mjmi} holds if  there is a constant $C$ such that 
 \begin{equation}\label{mjmi1}
\|F'[\gamma^\dagger](|\gamma-\gamma^\dagger|)\|_Y
\leq   C
\|F'[\gamma^\dagger](\gamma - \gamma^\dagger)\|_Y,\end{equation}
and 
\[ \|\gamma-\gamma^\dagger\|_{\infty} \leq C \thet . \]
\end{theorem} 
\begin{proof}
We have 
\begin{align*} 
\|F'[\gamma^\dagger](\gamma - \gamma^\dagger)\|_Y &\leq 
\|F'[\gamma^\dagger](\gamma - \gamma^\dagger) + (F(\gamma^\dagger) - F(\gamma))\|_Y
+ \| F(\gamma) - F(\gamma^\dagger)\|_Y\\
& \leq 
\|F'[\gamma^\dagger](\tfrac{|\gamma-\gamma^\dagger|^2}{\gamma})\|_Y 
+ \| F(\gamma) - F(\gamma^\dagger\|_Y \\
& \leq \|\tfrac{1}{\gamma}\|_\infty 
\|F'[\gamma^\dagger](|\gamma-\gamma^\dagger|^2)\|_Y 
+ \| F(\gamma) - F(\gamma^\dagger)\|_Y \\
& \leq \frac{1}{\ai} \thet 
\|F'[\gamma^\dagger](\gamma - \gamma^\dagger)\|_Y 
+ \| F(\gamma) - F(\gamma^\dagger)\|_Y. 
\end{align*}
Thus, 
 \[ (1 -\frac{\thet }{\ai}  )  \|F'[\gamma^\dagger](\gamma - \gamma^\dagger)\|_Y 
 \leq \| F(\gamma) - F(\gamma^\dagger)\|_Y
 \] 
Using  \eqref{babel0}, with $\gamma,\gamma^\dagger$ swapped 
and the constant 
  $C =   {2 + \xi_\dagger} = 2 + \|\tfrac{\gamma-\gamma^\dagger}{\gamma}\|_\infty $ 
we find, 
 \begin{align*} 
&  \|F(\gamma)- F(\gamma^\dagger) - F'[\gamma](\gamma^\dagger - \gamma) \|_Y \\
&\leq   \|F(\gamma)- F(\gamma^\dagger) - F'[\gamma^\dagger](\gamma^\dagger - \gamma) \|_Y+ 
\|  F'[\gamma^\dagger](\gamma^\dagger - \gamma)-  F'[\gamma](\gamma^\dagger - \gamma)\|_Y \\
& \leq \|F'[\gamma^\dagger] (\tfrac{|\gamma-\gamma^\dagger|^2}{\gamma} )\|_Y+
C \|F'[\gamma^\dagger] (\tfrac{|\gamma-\gamma^\dagger|^2}{\gamma}) \|_Y \\
& \leq 
(1+C)\|\tfrac{1}{\gamma}\|_{\infty}
 \|F'[\gamma^\dagger]{|\gamma-\gamma^\dagger|^2} \|_Y \leq 
(3+\xi) \frac{1}{\ai}
\thet \|F'[\gamma^\dagger]{\gamma-\gamma^\dagger} \|_Y \\
& \leq (3 +\xi) 
\frac{\frac{\thet }{\ai}}{1 -\frac{\thet }{\ai}}
 \| F(\gamma) - F(\gamma^\dagger)\|_Y.
\end{align*} 
Thus if 
\[(3 +\xi) 
\frac{\frac{\thet }{\ai}}{1 - \frac{\thet }{\ai}\|_\infty} \leq \eta \Leftrightarrow 
\frac{\thet }{\ai} \leq 
\frac{\eta}{3 +\xi+ \eta} \Longleftarrow 
\frac{\thet }{\ai} \leq 
\frac{\eta}{3 +1 +\eta}, 
\]
 then 
the strong tangential cone condition is verified.
The sufficiency of \eqref{mjmi1} follows from this  bound by 
using \eqref{uese} with $w_1 =  (\gamma-\gamma^\dagger)^2$ and 
$w_2 = \|\gamma-\gamma^\dagger\|_{\infty} |\gamma-\gamma^\dagger|$, and thus
\[ 
\|F'[\gamma^\dagger](\gamma-\gamma^\dagger)^2\|_Y \leq 
\|\gamma-\gamma^\dagger\|_{\infty}  \|F'[\gamma^\dagger]\left(|\gamma-\gamma^\dagger|\right)\|_Y. 
\] 
\end{proof}

 An immediate corollary is obtained for the in case that  $\gamma$ is below or above the 
true conductivity since such a  monotonicity  yields \eqref{mjmi}.
\begin{corollary}\label{cortwo}
Let $\gamma,\gamma^\dagger \in D(F)$,
with  $\|\gamma -\gamma^\dagger\|_{L^\infty} \leq \ai$.
Assume that  either 
\[ \gamma(x) \leq \gamma^\dagger(x) \quad \text{or} \quad  
\gamma(x) \geq \gamma^\dagger(x) \qquad \qquad  \forall x \in \Omega \text{ a.e.} \] 
Then for 
\[ \|\gamma -\gamma^\dagger\|_{\infty} \leq \thet,  
 \]
with $\thet$ as in Theorem~\ref{th:mjmi}, 
the strong tangential cone condition is satisfied with some given $\eta$. 
\end{corollary}  
\begin{proof}
Because of monotonicity we have that 
$|\gamma(x)-\gamma^\dagger(x)| = \pm (\gamma(x)-\gamma^\dagger(x))$,
such that the result follows from \eqref{mjmi1} with $C = 1$.  
\end{proof}

Another corollary of this theorem is that the cone conditions are satisfied if a source condition or a 
conditionals stability estimate holds.  
\begin{corollary} 
Let $X$ be a Hilbert space that is continuously embedded into $L^\infty$. 
Assume that a source condition 
\[ \gamma - \gamma^\dagger = (F'[\gamma^\dagger]^* F'[\gamma^\dagger])^\mu \omega \]
holds with $\mu > \frac{1}{2}$,  where $F'[\gamma^\dagger]^*$ is the adjoint in $X$. 
Then, for $\|\gamma - \gamma^\dagger\|_{\infty}$ sufficiently small, 
 the strong tangential cone condition holds for a given $\eta \leq \frac{1}{2}$. 
\end{corollary} 
\begin{proof} 
The left-hand side of \eqref{mjmi} is bounded by 
\[
\|F'[\gamma^\dagger](|\gamma-\gamma^\dagger|^2)\|_Y \leq L \|\gamma-\gamma^\dagger\|_\infty^2 \leq 
L \|\gamma-\gamma^\dagger\|_X^2, \qquad L = \|F'[\gamma^\dagger]\|_{2,L(L^\infty,Y)}\]
The source condition implies a stability estimate (see, e.g., \cite[p.~59]{EHN96})
\[ \| \gamma - \gamma^\dagger\|_X \leq \|\omega\|^\frac{1}{1+2\mu}  
\|F'[\gamma^\dagger] (\gamma - \gamma^\dagger)\|_Y^\frac{2\mu}{2 \mu+1}.  \]
Combining the inequalities yields 
\begin{align*} 
 &\|F'[\gamma^\dagger](|\gamma-\gamma^\dagger|^2) \|_Y\leq 
 L\|\omega\|^\frac{2}{1+2\mu}  \|F'[\gamma^\dagger] (\gamma - \gamma^\dagger)\|_Y^{\frac{4\mu}{2 \mu+1}-1}
\|F'[\gamma^\dagger] (\gamma - \gamma^\dagger)\|_Y \\
& \leq 
 L\|\omega\|^\frac{2}{1+2\mu}  (L \|\gamma - \gamma^\dagger)\|_\infty)^{\frac{2\mu-1}{2 \mu+1}}
\|F'[\gamma^\dagger] (\gamma - \gamma^\dagger)\|_Y .
\end{align*} 
If $\mu < \frac{1}{2}$ then the assumptions in  Theorem~\ref{th:mjmi} are verified for 
\[ 
\|\gamma - \gamma^\dagger\|_{\infty} \leq 
 \left( \frac{\thet}{L\|\omega\|^\frac{2}{1+2\mu}} \right)^\frac{2 \mu+1}{2\mu-1}, \qquad 
 \|\gamma - \gamma^\dagger\|_{\infty} < \ai, \]
with $\thet$ as given in this theorem. 
\end{proof} 

In particular, it follows that the tangential cone condition locally holds true  in 
a dense subset of $N(F'[\gamma^\dagger])^\bot$.

An interesting observation is that the left-hand side in 
\eqref{mjmi1} generates a (semi)-norm 
\[\|w\|_*:= \|F'[\gamma^\dagger] \left( |w| \right)\|_Y. \]
Indeed, take $z_1,z_2$ in $L^\infty$ arbitrary, then 
$0 \leq |z_1(x) + z_2(x)| \leq |z_1(x)| + |z_2(x)|$. 
Then,  the triangle inequality 
for $\|w\|_*$ follows from \eqref{uese}
with $w_1 =  |z_1 + z_2|$ and $w_2 =  |z_1| + |z_2|$ 
 and the triangle inequality for $\|.\|_{Y}$. 
Thus, \eqref{mjmi1} can be interpreted as 
an  equivalence  condition between two (semi)-norms.
(Note that the reverse inequality is easy to obtain). 
 The well-known norm equivalence in finite-dimensional spaces leads to the
following result: 
\begin{corollary}
Let $\gamma,\gamma^\dagger \in D(F)$ with 
$\gamma - \gamma^\dagger$ being in a finite-dimensional space $X_n$,
and let $F'[\gamma^\dagger]$ be injective on $X_n$. 
Then there is a dimension-dependent constant $C_n$ such that for all 
\[ \|\gamma - \gamma^\dagger\|_{\infty} \leq C_n \thet, \]
the strong tangential cone condition is satisfied for some given $\eta$. 
\end{corollary}

We may slightly generalize the monotonicity result. 
We denote the positive  and negative part of $\gamma-\gamma^\dagger$ by 
$$ (\gamma-\gamma^\dagger)^+(x) = \max\{\gamma(x)-\gamma^\dagger(x),0\} \qquad 
(\gamma-\gamma^\dagger)^- :=  -\min\{\gamma(x)-\gamma^\dagger(x),0\} .$$
such that 
\[ \gamma-\gamma^\dagger = (\gamma-\gamma^\dagger)^+-  (\gamma-\gamma^\dagger)^-.\]
 
\begin{theorem}\label{th:six}
Assume that there exists a constant $C$ or a constant $\nu <1$ such that 
\begin{align}
\label{fir}
 \|F'[\gamma^\dagger](\gamma-\gamma)^+(x) \|_Y &\leq 
C \|F'[\gamma^\dagger](\gamma-\gamma) \|_Y \quad \text{or}  \\
\label{fir1}
  \|F'[\gamma^\dagger](\gamma-\gamma^\dagger)^+(x) \|_Y 
&\leq \nu  \|F'[\gamma^\dagger](\gamma-\gamma)^-(x) \|_Y   
\end{align}
(or the respective inequalities with $()^+$ and $()^-$ swapped) hold. 

In case \eqref{fir}, if 
$$\|\gamma-\gamma^\dagger\|_\infty  \leq \thet\frac{1}{(2 C+1)}$$ 
holds, and  
in case \eqref{fir1}, if 
$$\|\gamma-\gamma^\dagger\|_\infty  \leq  \frac{\thet}{3} (1-\nu)$$
holds, with $\thet$ from Theorem~\ref{th:mjmi},  then 
\eqref{mjmi} is satisfied. 
In particular, under one of these conditions 
the tangential cone condition holds with any given $\eta$ if additionally 
 $\|\gamma-\gamma^\dagger\|$ is sufficiently small.  
\end{theorem}
\begin{proof}
Set $p(x) = (\gamma-\gamma^\dagger)^+(x)$ and $n(x) = (\gamma-\gamma^\dagger)^+(-)$, then 
$\gamma - \gamma^\dagger = p -n$ and $(\gamma - \gamma^\dagger)^2 = p^2 + n^2.$
 We have 
\begin{align*} & \|F'[\gamma^\dagger](|\gamma-\gamma^\dagger|) \|_Y\leq 
\|F'[\gamma^\dagger](p + n)\|_Y \leq 
\|F'[\gamma^\dagger] p\|_Y + \| F'[\gamma^\dagger] n\|_Y \\
&\leq 
 \|F'[\gamma^\dagger] p\|_Y + \| F'[\gamma^\dagger] (n -  p) \|_Y + \|F'[\gamma^\dagger] p\|_Y\\
 &  \leq 
2 C \|F'[\gamma^\dagger]( p-n)\|_Y + \| F'[\gamma^\dagger] (n -  p) \|_Y
\leq (2 C+1)  \|F'[\gamma^\dagger](\gamma-\gamma)\|_Y. \end{align*} 
Thus, 
\[ \|F'[\gamma^\dagger]((\gamma-\gamma)^2)\|_Y \leq 
\|\gamma-\gamma^\dagger\|_\infty (2 C+1)  \|F'[\gamma^\dagger](\gamma-\gamma)\|_Y 
\] 
 and \eqref{mjmi} holds.  
 In case of \eqref{fir1}, we estimate 
 \begin{align*}  &\|F'[\gamma^\dagger](\gamma-\gamma^\dagger)\|_Y = 
 \|F'[\gamma^\dagger] p - F'[\gamma^\dagger] n \|_Y \geq \left| \|F'[\gamma^\dagger]p\|_Y - 
 \|F'[\gamma^\dagger]n\|_Y\right| \\
 &\geq 
 (1- \nu) \|F'[\gamma^\dagger] p\|_Y\, .
 \end{align*}
 Hence \eqref{fir} holds with $C  = \frac{1}{1-\nu}$. The result then follows from 
 the first part with $\nu \geq 0$. 
\end{proof}

Finally, as the most constructive result, we establish the 
local tangential cone conditions for $C^2$-conductivities with ``unbalanced''   positive and negative
part.

\begin{theorem}
Assume that $\|\gamma-\gamma^\dagger\|_{C^2(\Omega)} \leq C_1$
and $\|(\gamma-\gamma^\dagger)\|_\infty \leq C_2 < \frac{\thet}{3} $
with $\thet$ from Theorem~\ref{th:mjmi} given some $\eta \leq 1$. 
There is  a nonnegative nondecreasing function $\psi$  such that if 
\begin{align} 
  \|(\gamma-\gamma^\dagger)^-\|_\infty &\leq \psi(\|(\gamma-\gamma^\dagger)\|_\infty)  \label{twotwo}
\end{align} 
holds (or with the roles of $+$ and $-$ swapped), then  the 
strong tangential cone condition is satisfied with this $\eta$.   
\end{theorem}
\begin{proof}
Without loss of generality we assume that 
$\|(\gamma-\gamma^\dagger)^+\|_\infty > \|(\gamma-\gamma^\dagger)-\|_\infty $ and 
hence that $\|(\gamma-\gamma^\dagger)\|_\infty = \|(\gamma-\gamma^\dagger)^+\|_\infty$
The case that that negative and positive part have equal norm is ruled out by the 
assumptions of the theorem.  
Let $x_0$ be a point in $\Omega$, where the maximum $M$ of 
$\delta \gamma: = (\gamma-\gamma^\dagger)^+$, 
 is attained. Then $\delta \gamma'(x_0) = 0$, 
and with the $C^2$-bound we may find an estimate 
\[ (\gamma-\gamma^\dagger)^+(x) \geq M-  \frac{C_1}{2}\|x-x_0\|^2 \]
Thus for $\|x-x_0\|^2 \leq \frac{M}{C_1}$, we have that 
$(\gamma-\gamma^\dagger)^+(x) \geq \frac{M}{2}$. By the monotonicity result in \eqref{uese}
 it follows 
that 
\[ \|F'[\gamma^\dagger](\gamma-\gamma)^+(x) \|_Y \geq 
\frac{\|(\gamma-\gamma^\dagger)^+\|_\infty }{2} 
\|F'[\gamma^\dagger]\chi_{B_\frac{M}{C_1}(x_0)} \|_Y \]
where $B_r(x_0)$ is the unit ball with center $x_0$ and radius $r$ and 
$\chi$ denotes the characteristic function. 
Define 
\[ \kappa(m):= \inf_{B_m(x_0) \subset \Omega}  \|F'[\gamma^\dagger]\chi_{B_m}(x_0) \|. \]
This defines a nonnegative and nondecreasing function. 
%
Thus, 
\[  \|F'[\gamma^\dagger](\gamma-\gamma)^+(x) \|_Y \geq
\frac{M}{2} \phi(\tfrac{M}{C_1}).  \]  
Define $\nu:= 1-\frac{3 C_2}{\thet}$. By assumption in the theorem, $0 < \nu < 1$. 
We let 
\[ \psi(M):= \frac{1}{2 L} M \kappa(\frac{M}{C_1}) \nu. \]
If \eqref{twotwo} is  satisfied with this $\psi$, then with $M = \|(\gamma-\gamma^\dagger)^+\|_\infty$, we have 
\begin{align*} 
&\|F'[\gamma^\dagger](\gamma-\gamma)^-(x) \|_Y \leq 
L \|(\gamma-\gamma)^-\|_{\infty} \leq 
\psi(\|(\gamma-\gamma)^+\|_{\infty}) \\ 
& \leq \nu
\frac{M}{2} \kappa(\frac{M}{C_1}) \leq \nu  \|F'[\gamma^\dagger](\gamma-\gamma)^+(x) \|_Y.
\end{align*} 
Thus \eqref{fir1} holds and since 
\[  \frac{\thet}{3}(1-\nu) = C_2 \geq \|(\gamma-\gamma)\|_{\infty},  \]
the result follows from Theorem~\ref{th:six} .
\end{proof} 
The relevance of this result is that we have the tangential cone condition 
satisfied in a $L^\infty$-ball if the positive (or negative ) part of 
the difference of conductivities dominates the negative (resp. positive) part.

\begin{remark}\rm 
The last theorem might explains to some extend the 
 convergence behaviour of the 
Landweber iteration for the impedance tomography problem. Fact is that the cone conditions 
are only required for the iterates $\gamma_k$ and the true conductivity $\gamma^\dagger$. 
In many numerical experiments, the initial value for the iteration is often chosen as being 
strictly below the true conductivity, for instance, if $\gamma^\dagger$ correspond to inclusions 
that have  higher conductivities than the background one, and naturally the initial 
values of the iteration are taken as that background. 
Then by the last theorem, respectively, Corollary~\ref{cortwo}, 
the cone condition is satisfied for the initial iterate, 
and the iteration will stay at least bounded.
 This will also be true for a certain number of 
the following iterations. 
 However, it is not guaranteed that this will hold for all iterations. 
In a lucky case, a monotonicity property will hold up to a stopping index, and then the 
iteration appear as convergent. However, in an unlucky case the cone condition might get 
violated with the effect that the iterates can diverge even though the stopping criteria is 
not yet met. This effect, may mistakenly be regarded as semiiteration, i.e., divergence by
data error, although this has nothing to do with noisy data.   

For a fair investigation of the convergence of the Landweber method, it would 
be interesting to start with a $\gamma$ that has values below and above of $\gamma^\dagger$ on a 
high number of regions, e.g., $\gamma_0 = \gamma^\dagger  + \text{highly oscillatory}$. 
An interesting question is  whether the iterates of the 
Landweber method would still remain bounded in this case. 
\end{remark}


\bibliographystyle{siam}
\bibliography{MainEIT}
\end{document}